\documentclass[10pt]{article}

\usepackage{amssymb}
\usepackage{amsfonts}
\usepackage{graphicx}
\usepackage{amsmath}

\usepackage{appendix}

\usepackage[colorlinks, linkcolor=black, anchorcolor=black,citecolor=black]{hyperref}

\usepackage{mathrsfs}
\usepackage{frcursive}

\usepackage{geometry}
\usepackage{color}

\numberwithin{equation}{section}

\newcommand{\R}{\mathbb{R}}

\newtheorem{theorem}{Theorem}[section]

\newtheorem{assumption}[theorem]{Assumption}

\newtheorem{corollary}[theorem]{Corollary}

\newtheorem{definition}[theorem]{Definition}
\newtheorem{example}[theorem]{Example}

\newtheorem{lemma}[theorem]{Lemma}

\newtheorem{remark}[theorem]{Remark}

\newenvironment{proof}[1][Proof]{\noindent\textit{#1.} }{\hfill \rule{0.5em}{0.5em}}

\begin{document}


\title{Spreading speeds for  prey-predator systems in a shifting environment: a short proof}



\author{ Zhucheng Jin\footnote{Email: jinzc@ustc.edu.cn} \\
 School of Mathematical Sciences, University of Science and Technology of China, \\
  Hefei, 230026, People's Republic of China \\
}
\date{}
\maketitle





\begin{abstract}
This note  is concerned with  the spreading speed of the predator component in  a  class  of prey-predator   reaction-diffusion systems  with  spatiotemporal heterogeneity  depending on  a moving variable.  The main difficulty is that the full system lacks  a direct comparison principle. 
We establish a pointwise estimate showing that the prey density is close to its carrying capacity   wherever the predator  density is sufficiently small.  This allows us to compare solutions of  the system with  those of  scalar  Fisher-KPP equations in shifting environments.  Consequently, we obtain an   explicit formula for the  spreading speed of the predator.   In particular, for nonmonotone shifting environments, locking and nonlocal pulling phenomena may occur.

\vspace{0.1in}

	\vspace{0.1in}\noindent \textbf{2020 Mathematical Subject Classification}. 35K45, 35K57, 92D25
\end{abstract}

\section{Introduction}
In this paper, we investigate   spreading speeds for the following general reaction-diffusion systems  in a shifting environment
\begin{equation}\label{Pb}
\begin{cases}
\partial_t u= d \partial_{xx} u  +  u f(x-c_1t, u,v),  & t>0, x\in \R,\\
\partial_t v= \; \partial_{xx} v + vg( x-c_1 t,  u,  v ),  & t>0, x\in \R, \\
u(0, x)=u_0(x), \;\; v(0, x)= v_0(x), & x\in \R,
\end{cases}
\end{equation}
where $c_1>0$ is the shifting speed of  the  environment and $d>0$ is the diffusion rate of the prey.   
A typical example  of \eqref{Pb}   is a Lotka-Volterra  prey-predator system.    
For notational simplicity, we normalize the diffusion coefficient of the predator to be one. Indeed, a system with diffusion terms $d_1\partial_{xx} u$ and $d_2 \partial_{xx}v$, where $d_1,d_2>0$, can be reduced to the present form by a standard  time rescaling argument.
Our main goal is to determine the spreading speed of the predator species when the prey  is initially uniformly established.  The detailed assumptions  of $f$ and $g$ are given in the  next section.


The concept of the  spreading speed was  introduced by Aronson and Weinberger \cite{aronson1978multidimensional}  in the context of  a scalar reaction-diffusion equation  to  describe the  spatial expansion of a solution front  over time.   The celebrated example is the   Fisher-KPP  equation in a homogeneous environment,
\begin{equation*}
\partial_{t} u=  \partial_{xx} u + u(r-u),    \quad \text{ for }   t>0, x\in \R.
\end{equation*} 
For the solution $u$ of this  Fisher-KPP  equation  with  compactly supported initial data, it follows from \cite{aronson1978multidimensional}  that there exists a constant $c_*:= 2\sqrt{r}$ such that 
\begin{equation}\label{eq-kpp}
\lim_{t\to +\infty } \sup_{ x \geq ct} u(t,x) =0,  \;\; \forall c>c_* \;\; \text{ and } \;\; \liminf_{t\to +\infty } \inf_{ 0\leq x \leq ct} u(t,x)> 0,  \;\; \forall c\in (0,c_*).  
\end{equation}
The constant $c_*$ is called  the spreading speed.   In \eqref{eq-kpp}, the speed $c_*$   is linearly determined. 
Since then, the  spreading speed of various reaction-diffusion equations/systems and their  generalizations in heterogeneous environments have been extensively studied, see for example  \cite{berestycki2019asymptotic, ducrot2019spreading,  girardin2019invasion} and the  references  therein.     

Shifting environments arise naturally in population dynamics when the spatial distribution of favorable  habitats changes over time, for instance due to  climate warming or other large-scale environmental changes.  Such models   were  introduced  in   \cite{potapov2004shift}  and \cite{berestycki2009shift}.  
 To model such effects, a simple way is to  assume that  the growth or interaction rates depend on the moving variable $x-c_1t$,  where $c_1$ is the velocity of the  moving habitat.   Beyond  biological modeling, the shifting  heterogeneity can also  arise  in a triangular system that  is partially decoupled, see \cite{fang-lou-wu2016,holzer2014accelerated} and the discussion in Section 2.1 of  \cite{GGM2024}.   The  propagation dynamics of equations/systems in shifting environments have been extensively  studied in   \cite{choi2021persistence, fang-lou-wu2016, GGM2024,  LamLee2024pp, LNY2024, LY2022, Libingtuan2014}.  See also \cite{Yi-Zhao2020jfa, Yi-Zhao2026}   for the general theory  on these problems.    We refer the reader to the  survey article \cite{wang2022shiftSurvey} and the references therein.  
 Unlike  homogeneous media,  the spreading speed may  depend on the velocity and the profile of the moving habitat.  As shown in  \cite{berestyckiFang2018,  LY2022} and \cite[Chapter 6]{LamLou2022},   some intriguing  phenomena may occur, for instance, locking and nonlocal pulling phenomena, where the spreading speed is not linearly determined.  
The  locking phenomenon means that the spreading speed of the species equals    the shifting speed, while the nonlocal pulling  phenomenon  means that the spreading speed  is not linearly determined by the leading edge, which  can be enhanced by the presence of a good habitat at large spatial scales.

More  recently, \cite{GGM2024} and \cite{LNY2024} give a complete characterization of the spreading speed for scalar KPP equations in shifting environments with  nonmonotone growth rate functions. 
The nonmonotone shifting environments  are significant  in biological modeling, since many biological processes depend  unimodally on  temperature \cite{deutsch2008impacts}.    The lack of  monotonicity  in the shifting environments introduces additional difficulties in characterizing the propagation dynamics.

For prey-predator systems, the analysis becomes even more delicate. The full system generally lacks a  comparison principle, and the growth rate of the predator depends on the prey density, which is affected by predation.    Hence, the methodology for  scalar equations  cannot be applied directly.   For  prey-predator systems in homogeneous environments, the spreading speed was studied in \cite{ducrot2021asymptotic, ducrot2019spreading}. Recently  the propagation dynamics of  prey-predator systems in shifting or nonautonomous  environments were investigated by \cite{choi2021persistence, LamLee2024pp, ducrotjin2023systemspreading}.      

The main purpose of this note is to  provide a  reduction from the prey-predator system to an associated scalar Fisher-KPP equation in a shifting environment, where the  environmental profile may be nonmonotone.   The local favorable patches may affect the selection of  speed.   This reduction allows us to  characterize how a nonmonotone shifting environment affects the predator invasion speed. This reduction is obtained through  a pointwise estimate showing that the prey density is  close to its carrying capacity wherever the predator density is small, see Lemma \ref{Lem-estimate} below.    

The pointwise estimate was introduced in  \cite{ducrotjin2023systemspreading} for time heterogeneous prey-predator systems.   In contrast,  the present problem contains a moving spatial heterogeneity, and the speed selection is governed by the interaction between the predator front and the shifting environmental profile.   We prove   that the predator component has exactly the same spreading speed as the scalar KPP equation with effective growth rate   $g(x-c_1t, 1, 0)$. Here, we emphasize that $g(x-c_1t, 1,0)$ can be nonmonotone. 
The locking and nonlocal pulling mechanisms discovered for scalar KPP equations in shifting environments persist at the level of predator invasion in a non-cooperative  prey-predator system.
Our main result  also confirms  the conjecture about  the locking case stated in \cite[Remark 3]{LamLee2024pp}, where a  Hamilton-Jacobi approach was used  to  study  a Lotka-Volterra prey-predator model with a monotone conversion rate function.

Finally,  we mention the work of Bramson \cite{bramson1978maximal},  which employs  probabilistic  arguments  to establish   the logarithmic delay in  the position of  solutions to  Fisher-KPP equations relative to the position of the traveling front with minimal wave  speed.  For prey-predator systems,    characterizing  such a   correction  remains  a challenging open problem. 
 The argument used  in this note does not seem applicable to addressing  this issue.

The rest of  this paper is organized as follows. In Section 2, we present the assumptions and state our main results.   In Section 3, we recall the spreading speed results of KPP equations in shifting environments and prove  some basic results. In Section 4, we prove  the spreading speed  theorem.  In 
Section 5, we prove the asymptotic behaviors of our problem and display a concrete example.

\section{Assumptions and main results}

\subsection{General assumptions}
Now, we  state the  assumptions on $f$ and $g$.
\begin{assumption}\label{Ass-fg}
Assume that $f: \R \times [0, \infty)\times  [0, \infty) \to \R$ and $g: \R \times [0, \infty)\times  [0, \infty) \to \R$   satisfy: 
\begin{enumerate}
\item[(i)]  For every $R>0$,  $z \mapsto f(z,u,v)$ and $z \mapsto g(z,u,v )$ are bounded and uniformly continuous on $\R$, uniformly for $(u,v) \in [0, R]^2$.  Moreover,   $f$ and  $g $ are Lipschitz continuous in $(u,v)\in [0,R]^2$,  uniformly  with respect to $z\in \R$;
\item [(ii)]  (\textbf{Predation type}) For each $u>0$, $z\in \R$, the  map  $v\mapsto f(z,u, v)$ is decreasing. For each $v>0$,  $z\in \R$, the map $u\mapsto g(z,u,v)$  is increasing.
\item[(iii)] (\textbf{Monostable}) $f(\cdot, 1,0) \equiv 0$ and $\inf_{z\in \R} f(z,u, 0)> 0$ for all $u\in [0, 1)$. 
\item [(iv)] (\textbf{Strong KPP}) For each $u\geq 0$,  the map  $v \mapsto g(z, u, v)$ is nonincreasing uniformly in $z\in \R$.  
\item [(v)] The limits $\lim_{z\to \pm \infty} g(z, 1,0 )$ exist  and 
$$\inf_{z\in \R} g(z, 1, 0) >0. $$

\end{enumerate}
\end{assumption}
\begin{remark}
\rm 
In Assumption~\ref{Ass-fg}~\textit{(iii)}, we normalize  the predator-free  carrying capacity of the prey to the constant state $1$.   This normalization is used  in  Lemma~\ref{Lem-estimate}, where we  show that the prey density is close to $1$ wherever the predator density is  small. 
If the   predator-free prey state is nonconstant, then one would  need to  investigate the corresponding positive entire solution $p^*(t,x)$ of the predator-free prey equation.   The effective growth rate of the predator  would be  expected to   involve  $g(x-c_1 t, p^*(t,x), 0)$  rather than   $g(x-c_1t, 1, 0)$.  However, this natural extension  requires  additional assumptions and a separate analysis.  It is therefore beyond the scope of the present  note.

 Denote
\begin{equation}\label{rrr}
r^{\pm} := g(\pm \infty , 1, 0).
\end{equation}
 By Assumption \ref{Ass-fg} \textit{(v)},   we have  $r^\pm >0$.  In this  note,  it is not necessary to  require $z\mapsto g(z, 1,0)$ to be monotone. We are interested in the case 
$$\sup_{z\in\R} g(z,1,0) > \max\{ r^+, r^- \}. $$

\end{remark}

\vspace{1em}

Next  we impose the  dissipative assumption. 
\begin{assumption} \label{Ass-bound}
Assume that there exist constants $\overline{V}>0$  and $\delta>0$ such that
$$  \sup_{z\in\R} g(z,1,\overline{V})\leq 0   \quad
\text{ and  } \quad
\inf_{\substack{z\in \R, \; 0\leq u\leq \delta}} f(z,u, \overline {V})> 0.
$$
\end{assumption}

\begin{assumption}\label{Ass-IC}
Let  $\overline{V}$ be chosen  as in Assumption \ref{Ass-bound}.  Let $u_0$ and $v_0$ be two bounded and continuous functions satisfying
$$
0<\inf_{x\in\R} u_0(x)\le u_0(x) \le 1,
\quad
0\le v_0(x)\le \overline{V},\quad
v_0\not\equiv 0,  \quad \text{ for all }   x\in \R. 
$$
and suppose that $v_0$ is compactly supported. 
\end{assumption}

\begin{remark}
\rm 
Assumption \ref{Ass-bound}  provides an admissible upper bound for the predator population and guarantees the persistence of the prey component at low density. Together with the monotonicity assumptions, it yields the invariant region stated in Lemma \ref{Lem-bound}    below.
\end{remark}

\subsection{Main results}
We consider only  rightward propagation, since leftward propagation  can be treated by    the change of  variable  $x\mapsto -x$.  We introduce the definition of rightward spreading speed, which will   be referred to   simply as  the spreading speed.
\begin{definition}\label{Def-speed}
Let $(u, v)$ be a  solution of \eqref{Pb}   with  initial data $(u_0, v_0)$ satisfying Assumption \ref{Ass-IC}. The quantity $c^*>0$ is called the \textit{rightward spreading speed} of $v$ if 
\begin{align*}
\lim_{t\to +\infty} \sup_{x\geq ct} v(t,x) =0,  \qquad & \forall c>c^*, \\
\liminf_{t\to +\infty} \inf_{ 0 \leq x\leq ct} v(t,x)>0,  \qquad & \forall c\in (0, c^*).
\end{align*}
\end{definition}

To state our main theorem, we first introduce an eigenvalue problem, 
\begin{equation}\label{eigen-eq}
\partial_{zz} \phi + g(z, 1,0 ) \phi = \Lambda \phi,  \quad \forall z\in \R.
\end{equation}
Let us define the generalized  principal  eigenvalue $\Lambda_1$ of \eqref{eigen-eq} as follows: 
\begin{equation}\label{ev}
\Lambda_1 := \Lambda_1(g)= \inf\left\{ \Lambda \in \R: \;\;  \exists \psi \in C^2_{loc} (\R), \psi>0,  \partial_{zz} \psi + g(z, 1,0) \psi \leq \Lambda \psi \text{ in } \R  \right\}.
\end{equation}
Several notions of generalized principal eigenvalues are analyzed in \cite{berestyckiRossi2015generalizations}.
We refer  to \cite{berestyckiRossi2015generalizations}, \cite[Proposition 3.1]{GGM2024} and  \cite[Proposition 4.2]{LNY2024} for some basic properties of $\Lambda_1$.

Now we  state our main theorem. 
\begin{theorem}\label{Thm1}
Let Assumptions \ref{Ass-fg}, \ref{Ass-bound} and \ref{Ass-IC} be satisfied and $c_1>0$.  Let $(u,v)$ be the solution of  \eqref{Pb}.   Then the  spreading speed of $v$ is given by 
\begin{equation*}
c_v^*= \begin{cases}
2\sqrt{r^+}, & \text{ if  }  0<c_1 \leq 2 \sqrt{r^+}, \\
c_1,   & \text{ if }  2\sqrt{r^+} \leq c_1 \leq  2 \sqrt{\Lambda_1}, \\
\frac{c_1}{2} - \sqrt{\Lambda_1 - r^{-} }  + \frac{r^{-}}{\frac{c_1}{2} - \sqrt{\Lambda_1 - r^{-} }  },   \quad  & \text{ if  }   2 \sqrt{\Lambda_1}  \leq  c_1 \leq 2 (  \sqrt{r^-} + \sqrt{\Lambda_1 -r^{-}}),\\
2\sqrt{r^-},    & \text{ if  } c_1  \geq   2 (  \sqrt{r^-} + \sqrt{\Lambda_1 -r^{-}}),
\end{cases}
\end{equation*}
 where $r^{\pm}$ and $\Lambda_1$ are defined in \eqref{rrr} and \eqref{ev} respectively.   
\end{theorem}
\begin{remark}
Note that the map $c_1 \mapsto c_v^*$  is nonmonotone.
When $ c_1 \in  (2\sqrt{r^+},  2\sqrt{\Lambda_1})$, the  predator propagates at  the same velocity as  the moving habitat, namely, the locking phenomena. However, for $ c_1 \in \big( 2 \sqrt{\Lambda_1},  2 (  \sqrt{r^-} + \sqrt{\Lambda_1 -r^{-}}) \big)$,  the nonlocal pulling  phenomenon  occurs,    where  distant  favorable  regions enhance the  spreading  speed.    
\end{remark}

\begin{remark}
\rm
If $\Lambda_1=r^{+} >r^{-}$, then the second  case (locking)  in Theorem \ref{Thm1} is eliminated. This situation has been investigated in \cite{LamLee2024pp}.    In case that $\Lambda_1=r^{-} >r^{+}$, the  third  case (nonlocal pulling) is eliminated.  When $\Lambda_1=r^{-} =r^{+}$, both the second and the third cases are eliminated.
\end{remark}

\subsection{Asymptotic behaviors}
As a consequence of Theorem \ref{Thm1}, 
we have the asymptotic behavior of solutions  at the leading edge. 
 \begin{corollary}\label{Cor1}
 Let the assumptions of Theorem \ref{Thm1} be satisfied.  Then we have
\begin{equation}\label{11}
\lim_{t \to +\infty} \sup_{ x \geq (c_v^*+ \eta)t} \big\{ \left| u(t,x)-1\right| + \left| v(t,x)\right|     \big\} =0,    \forall \eta >0.
\end{equation}
 \end{corollary}

 The spreading speed result above does not by itself determine the asymptotic state behind the invasion front. To describe the convergence in the interior spreading region, we impose additional assumptions on the limiting ODE systems, which are close to  \cite[Assumption 2.5]{ducrot2019spreading}. 
 
\begin{assumption}\label{Ass-limit}
Assume the following limits  exist:
\begin{equation*}
\lim_{z\to \pm \infty} f(z, u, v) = f_\pm (u, v)  \text{ and }  \lim_{z\to \pm \infty} g(z, u, v) = g_\pm (u, v) \;\;  \text{ locally uniformly for } (u, v) \in [0, \infty) \times [0, \infty). 
\end{equation*}
For the limiting system, we impose that the vector fields $(f_{+},  g_+)$  (respectively $(f_-,  g_-)$) has a unique zero  $(u^*_+, v^*_+)$ (respectively $(u^*_-, v^*_-)$) in  the rectangle $[\beta, 1) \times (0, \overline{V}]$, where  $\beta:= \min\{\inf_{x\in\R} u_0(x), \delta \}$ and $\delta$ and  $\overline{V} $ are given in Assumption \ref{Ass-bound}.
\end{assumption}
  In general, the asymptotic behavior behind the front is not clear, which may depend to a large extent  on the dynamical  behavior of the following  limiting  ODE systems    
 \begin{equation}\label{ode1}
 \begin{cases}
 u_t= u f_+(u, v),\\
 v_t= vg_+(u,v),
 \end{cases}
 \end{equation} 
 and 
  \begin{equation}\label{ode2}
  \begin{cases}
  u_t= u f_-(u, v), \\
  v_t= vg_-(u,v).
  \end{cases}
  \end{equation}

 \begin{assumption}\label{Ass-LF}
 Assume  that  \eqref{ode1} and \eqref{ode2} admit   a strict Lyapunov function.    More precisely,  suppose that
 \begin{itemize}
 \item[(i)]  There exist strictly convex functions
  $$
  \Phi_{+},\Phi_{-}\in C^2\big([\beta,1)\times(0,\overline V],\mathbb R\big),
  $$
  such that \(\Phi_{+} \) attains its minimum at \((u_+^* , v_+^*) \), while \(\Phi_{-}\) attains its minimum at \((u_-^*,v_-^*)\).  And the functions  $\Phi_{+}$ and $\Phi_{-}$ satisfy
 \begin{equation*}
   (uf_+(u,v),  vg_+(u, v)) \cdot \nabla \Phi_{+} (u, v) \leq 0,  \;\; \forall (u, v) \in [\beta, 1)\times (0, \overline{V}],
 \end{equation*}
and 
 $$  (uf_-(u,v),  vg_-(u, v)) \cdot \nabla \Phi_{-} (u, v) \leq 0,   \;\;  \forall (u, v) \in [\beta, 1)\times (0, \overline{V}]. $$

\item[(ii)] Let   $( u^+(t),v^+(t) )$  and $( u^-(t),v^-(t) )$ be the  solutions of  \eqref{ode1} and \eqref{ode2}   with initial data $(\hat{u}_0,  \hat{v}_0) \in [\beta, 1)\times (0, \overline{V} ]$, respectively.  For   $i\in\{+, - \}$.   For each $(\hat{u}_0, \hat{v}_0) \in [\beta, 1)\times (0, \overline{V} ]$, assume that  we have: \\
if $\Phi_{i} (u^{i}(t), v^{i}(t) ) = \Phi_{i} (\hat{u}_0, \hat{v}_0)$  for all  $t\geq 0$, then 
$(\hat{u}_0, \hat{v}_0) = (u^*_{i}, v^*_{i}).  $

  \end{itemize}

 \end{assumption}

Then we have the following convergence result.
\begin{corollary}\label{Cor2}
Let the assumptions  of Theorem \ref{Thm1} and Assumptions \ref{Ass-limit} and \ref{Ass-LF}  be satisfied. 
 Let $(u,v)$ be the solution of  \eqref{Pb}.   Assume $d=1$.  Then   the following statements  hold true:  
\begin{itemize}
\item[(i)]  if $c_1\geq  2\sqrt{r^+}$, then for all  $\eta\in (0,  c_v^*/2)$,
\begin{equation}\label{22}
\lim_{t \to +\infty} \sup_{ 0\leq x\leq (c_v^*- \eta) t}\big\{ \left| u(t,x)-u^*_{-}  \right| + \left| v(t,x) - v^*_{-}\right|     \big\} =0;
\end{equation}
\item [(ii)] if $0<c_1< 2\sqrt{r^+}$,  then for all  $\eta\in (0, c_1/2)$,
\begin{equation}\label{33}
\lim_{t \to +\infty} \sup_{ 0 \leq x \leq (c_1- \eta) t}\big\{ \left| u(t,x)-u^*_{-}  \right| + \left| v(t,x) - v^*_{-}\right|     \big\} =0,
\end{equation}
and for all $\eta\in (0, (c_v^* -c_1) / 2)$, 
\begin{equation}\label{44}
\lim_{t \to +\infty} \sup_{ (c_1+ \eta ) t \leq x \leq (c_v^*- \eta) t}\big\{ \left| u(t,x)-u^*_{+}  \right| + \left| v(t,x) - v^*_{+}\right|     \big\} =0.
\end{equation}
 \end{itemize}

\end{corollary}

\begin{remark} 
\rm
From the formula of $c_v^*$ given in  Theorem \ref{Thm1}, we have the   fact:  
\begin{equation*}
 \begin{cases}
 c_v^*= c_v^*(c_1) 
 \leq c_1   &\text{ if } c_1\geq 2\sqrt{r^+},\\
  c_v^*= c_v^*(c_1) >c_1    &\text{ if } c_1 <2\sqrt{r^+}.
\end{cases}
\end{equation*}
Corollary \ref{Cor2}  shows that, in the invaded region of the predator,   the solution converges to the positive coexistence equilibrium selected by the limiting environment seen in the moving frame  $ x-c_1t$.  The asymptotic state depends on the relative position between the predator front ($x=c_v^* t$) and the shifting environmental interface ($x=c_1 t$). 

The assumption $d=1$ is  used only in the Lyapunov argument for the limiting systems. It can be removed in special cases; see example \eqref{Pb-ex}.

This corollary does not describe the  fine asymptotics  near $x=c_1t$ or  $x=c_v^*t$. We refer to \cite{lam-wu2025logarithmic} for  the logarithmic correction results  of KPP equations in shifting environments.
\end{remark}

\section{Preliminaries}
\subsection{Scalar KPP equations in shifting environments}
We first recall  a very recent result   
on the spreading speed of the scalar Fisher-KPP equation in a shifting environment, established by \cite{GGM2024,LNY2024}. 
Let us focus on the following equation
\begin{equation}\label{KPP}
\begin{cases}
\partial_t w = \partial_{xx} w + w \left(g(x-c_1 t) -w  \right), &t>0, x\in \R, \\
w(0, x)= w_0(x), &x\in\R,
\end{cases}
\end{equation}
where $c_1 \in \R$ and the initial function $w_0$ is   nonnegative, nontrivial, bounded, continuous and compactly supported on $\R$.  
Assume that the function $g$ is continuous  and  positive on $\R$ with 
\begin{equation*}
g^{\pm}:=  g(\pm \infty) >0.
\end{equation*} 
Similarly to \eqref{eigen-eq} and \eqref{ev}, we define $\lambda_{1,g}$ as the generalized principal eigenvalue of the  following eigenvalue problem:
\begin{equation*}
\phi_{yy} + g(y) \phi = \lambda \phi,  \quad \forall  y\in \R.
\end{equation*}
Namely, 
\begin{equation*}
\lambda_{1,g}= \inf\left\{ \lambda \in \R: \;\;  \exists \psi \in C^2_{loc} (\R), \psi>0,  \psi'' +g(y) \psi \leq \lambda \psi \text{ in } \R  \right\}.
\end{equation*}
With these notations,   the exact spreading speed $c^*$  of \eqref{KPP} is characterized  by  the  following theorem. Note that the spreading speed of \eqref{KPP}   is completely  determined by  quantities $g^{+}$, $g^{-}$, $c_1$ and $\lambda_{1,g}$.  Different approaches to its proof can be found in \cite{GGM2024, LNY2024}. The dependence  of the spreading speed   $c^*$ on $c_1$ has been illustrated  in \cite{GGM2024, LNY2024} for various scenarios.
\begin{theorem}[\cite{GGM2024, LNY2024}]\label{Thm-KPP}
Let $w$ be a solution of \eqref{KPP}. Then the rightward  spreading speed of $w$ is given by 
\begin{equation*}
c^*= \begin{cases}
2\sqrt{g^+}, & \text{ if } c_1 \leq 2 \sqrt{g^+}, \\
c_1,   & \text{ if } 2\sqrt{g^+} \leq c_1 \leq  2 \sqrt{\lambda_{1,g}}, \\
\frac{c_1}{2} - \sqrt{\lambda_{1,g}- g^{-} }  + \frac{g^{-}}{\frac{c_1}{2} - \sqrt{\lambda_{1,g}- g^{-} }  },   \quad  & \text{ if }   2 \sqrt{\lambda_{1,g} }  \leq  c_1 \leq 2 (  \sqrt{g^-} + \sqrt{\lambda_{1,g}-g^{-}}),\\
2\sqrt{g^-},    & \text{ if } c_1  \geq   2 (  \sqrt{g^-} + \sqrt{\lambda_{1,g}-g^{-}}).
\end{cases}
\end{equation*}
\end{theorem}
\begin{remark} \label{Rem-kpp}
\rm

The piecewise constant case was  considered  in \cite{GGM2024}, where the spreading speed was derived by constructing suitable supersolutions and subsolutions.   The continuous case was treated in  \cite{LNY2024}   by using a  Hamilton-Jacobi approach with a dynamic   junction condition. 
Moreover,  by  \cite[Theorem 2.13]{LNY2024},  the same conclusion also holds for general KPP reaction terms.

\end{remark}

\subsection{Invariant region and a Liouville-type lemma}
Now we prove that the solution of \eqref{Pb} is  invariant in a rectangle. 
\begin{lemma}\label{Lem-bound}
Let Assumptions \ref{Ass-fg}, \ref{Ass-bound} and \ref{Ass-IC} be satisfied. Denote $\beta:= \min\{  \inf_{x\in\R } u_0(x),   \delta \}$.  Then  
$$\beta \leq  u(t,x)\leq 1,   \;\; 0 \leq v(t,x) \leq   \overline{V}, \quad  \forall  t\geq 0, x\in \R. $$  
\end{lemma}
\begin{proof}
 Although the full system does not admit  a comparison principle, scalar comparison can still be applied to each component separately.  We first observe that $0$ is a subsolution for each component equation.
 Since $0$ satisfies $v$-equation for any $u$   and initial function $v_0\geq 0$, we obtain $v(t,x) \geq 0$ for all $t\geq 0$ and $x\in \R$.  Similarly,  we have $u(t,x) \geq 0$ for all $t\geq 0$ and $x\in \R$.    From the monotonicity of $f$ and $g$, we have 
 $$   f(z, 1, v)  \leq  f(z,1,0) =0 \quad  \forall z\in \R, v\geq 0  \;  \text{ and }  \; g(z, u, \overline{V}) \leq g(z,1, \overline{V}) \leq 0 \quad \forall z\in\R,  u \in [0,1].$$
Hence, $u=1$ is a supersolution of $u$-equation for any $v\geq 0$ and $v= \overline{V}$ is a supersolution of $v$-equation for any $u \in [0,1]$.  Together with the initial condition in Assumption \ref{Ass-IC},  this yields $ (u, v)(t,x) \in [0, 1] \times [0, \overline{V}]$ for all $t\geq 0$  and $x\in \R$.

Next we prove that $u$ has  a positive lower bound for positive initial data.  By Assumptions \ref{Ass-fg} and \ref{Ass-bound}, the constant function  $u=\beta $  satisfies 
\begin{equation*}
\partial_t \beta  -\partial_{xx}  \beta \leq  \beta f(x-c_1 t,\beta, \overline{V} ) \leq  \beta f(x-c_1 t ,\beta, v ) ,  \quad \forall t\geq 0,  x\in \R, v\in [0, \overline{V}].
\end{equation*}
Thus, the comparison principle for scalar equation implies that $u\geq \beta>0$.  The proof is complete.

\end{proof}

Next we  state a Liouville-type lemma, which   plays an  important role in the  proof of pointwise estimate. 
\begin{lemma}\label{Lem-liou}
Let $u$  be a bounded entire solution of
 \begin{equation*}
  \partial_t u= d \partial_{xx} u +  u \tilde{f}(x-c_1 t, u, 0), \quad  (t,x) \in \R^2.
 \end{equation*} 
where   $\tilde f(\cdot, u, 0)$  is a locally uniform limit of translations of  $f(\cdot, u, 0)$
and satisfies the monostable condition inherited from Assumption \ref{Ass-fg}\textit{(iii)}. 
If
\[
0<\beta\leq u(t,x)\leq 1 \quad \text{in } \mathbb R^2,
\]
then $u \equiv 1$.
\end{lemma}
\begin{proof}
Denote  $w(t,x) := u(t, x+c_1 t)$ and  $m: = \inf_{(t,x) \in \R^2}w(t,x)$.  Assume  $m<1$.  There exists a sequence $\{ (t_n,x_n) \}_n \subset  \R^2$ such that  $w(t_n, x_n) \to m$ as $n \to +\infty$.   Define $w_n (t,x) := w(t+t_n, x+x_n)$.  By parabolic estimates, up to a subsequence,  it follows that  
\begin{equation*}
\begin{split}
\lim_{n\to +\infty}  w_n (t,x) = \tilde{w} (t,x)   \quad  &\text{ locally uniformly in } (t,x) \in \R^2, \\
 \lim_{n\to +\infty} \tilde{f}(x+x_n, w_n, 0) = \hat{f} (x, \tilde{w}, 0)   \quad  &\text{ locally uniformly in } (t,x) \in \R^2,
\end{split}
\end{equation*}
and  $\tilde{w}$ satisfies 
\begin{equation}\label{wtilde}
\partial_{t} \tilde{w} = d \partial_{xx} \tilde{w} +  c_1 \partial_x \tilde{w} + \tilde{w} \hat{f} (x, \tilde{w}, 0). 
\end{equation}
Note that   $\tilde{w}(0,0) = m  \in [\beta, 1)$ and $(0, 0) $ is the global minimum point.
Thus we have  
$$\partial_{t} \tilde{w}(0,0) =0, \;\;  \partial_{x} \tilde{w}(0,0) =0  \;\; \text{ and }  \;\;   \partial_{xx} \tilde{w}(0,0) \geq 0.  $$
By Assumption \ref{Ass-fg}\textit{(iii)}, it follows   that  $\inf_{x\in\R} \hat{f} (x, m, 0)>0$.  This contradicts  \eqref{wtilde}. We have  $m= 1$, namely $w\equiv 1$, and hence $u\equiv 1$. This proves the lemma.
 
\end{proof}

\section{Proof of the spreading speed}
To prove Theorem \ref{Thm1},  we first derive the  following  pointwise estimate, which  plays an important role in the lower estimate of spreading speed of $v$.

\begin{lemma}\label{Lem-estimate}
Let the assumptions of  Theorem \ref{Thm1} be satisfied.   For each $\alpha>0$, there exist  constants  $M_{\alpha}>0$ and $T_{\alpha} >0$ such that 
\begin{equation*}
1-u(t,x) \leq \alpha+ M_{\alpha} v(t,x), \quad \forall t\geq T_{\alpha}, \forall x\in\R.
\end{equation*}
\end{lemma}
\begin{proof}
Argue by contradiction and suppose that there exist  sequences $\{x_n\}_{n\geq 1} \subset \R $,  $\{t_n\}_{n\geq 1} \subset \R$ with  $t_n \to +\infty$ as $n \to +\infty$,     and some constant  $\alpha_0>0 $  such that 
\begin{equation} \label{contra1}
1-u(t_n, x_n) > \alpha_0 + n  v(t_n, x_n), \;\;\;\; \forall n\geq 1.
\end{equation}
Set
\begin{equation*}
u_n(t,x) := u(t+t_n, x+x_n) \text{ and } v_n(t,x) := v(t+t_n, x+x_n). 
\end{equation*}
Note that $(u_n, v_n) (t,x)$ satisfies 
\begin{equation*}
\begin{cases}
\partial_t u_n= d \partial_{xx} u_n +   u_n f \big( x+x_n -c_1(t+t_n), u_n, v_n \big), & \forall t>-t_n, x\in \R, \\
\partial_t v_n =  \partial_{xx} v_n   + v_n  g\big(  x+x_n -c_1(t+t_n), u_n,   v_n  \big), & \forall t>-t_n, x\in \R.
\end{cases}
\end{equation*}
By parabolic estimates and Assumption \ref{Ass-fg}\textit{(i)}, up to a subsequence, we have 
\begin{align*}
\lim_{n\to +\infty} u_n(t,x) = \tilde{u}(t,x) \quad \text{ locally uniformly in }  (t,x) \in \R^2, \\
\lim_{n\to +\infty} v_n(t,x) = \tilde{v}(t,x) \quad  \text{ locally uniformly in }  (t,x) \in \R^2,
\end{align*}
and 
\begin{align*}
\lim_{n\to +\infty} f \big( x+x_n -c_1(t+t_n), u_n, v_n \big) = \tilde{f} \big( x-c_1t, \tilde{u}, \tilde{v} \big) \quad \text{ locally uniformly in }  (t,x) \in \R^2, \\
\lim_{n\to +\infty} g \big( x+x_n -c_1(t+t_n), u_n, v_n \big) = \tilde{g} \big( x-c_1t, \tilde{u}, \tilde{v} \big) \quad \text{ locally uniformly in }  (t,x) \in \R^2,
\end{align*}
The parabolic estimates ensure that    $(\tilde{u}, \tilde{v}) (t,x)$  satisfies  for all $(t,x) \in \R^2$,
\begin{equation*}
\begin{cases}
\partial_t \tilde{u}= d \partial_{xx} \tilde{u}  +  \tilde{u}  \tilde{f} \big( x-c_1t, \tilde{u}, \tilde{v} \big),\\
\partial_t \tilde{v}=  \partial_{xx} \tilde{v} + \tilde{v} \tilde{g} \big( x-c_1t, \tilde{u}, \tilde{v} \big).  \\
\end{cases}
\end{equation*}
Letting $n \to +\infty $ in  \eqref{contra1} and using the boundedness of $u$,  we  obtain
$\tilde{v}(0, 0) = 0$. The  strong maximum principle   applied to  the $\tilde{v}$-equation  implies that $\tilde{v}\equiv 0$.  Thus, $\tilde{u}$ satisfies 
\begin{equation*}
\partial_t \tilde{u}= d \partial_{xx} \tilde{u}  + \tilde{u} \tilde{f} \big( x-c_1t, \tilde{u}, 0 \big) ,     \quad \forall (t,x ) \in \R^2.
\end{equation*}
Moreover,  Lemma \ref{Lem-bound} gives $\tilde{u} \geq \beta >0$.   
Then  the Liouville-type  Lemma \ref{Lem-liou}  gives  $\tilde{u} \equiv 1$.     However, \eqref{contra1} yields 
$1- \tilde{u}(0, 0) \geq \alpha_0 >0$.
This contradiction completes  the proof.

\end{proof}

We now prove Theorem \ref{Thm1}  using  Lemma \ref{Lem-estimate} and Theorem \ref{Thm-KPP}.

\begin{proof}[Proof of Theorem \ref{Thm1}]
\textbf{Upper estimate:}
 We  first  show the upper estimate of the spreading speed of $v$.  Since  $u\in [\beta, 1]$,  the solution $v$ satisfies the following differential inequality
\begin{equation*}
\partial_t  v \leq \partial_{xx} v + v g(x-c_1t, 1, v) \quad \forall t>0, x\in \R.
\end{equation*}
Note that  the reaction term   $vg(\cdot, 1, v)$ satisfies the  KPP condition.
The comparison principle, together with Theorem \ref{Thm-KPP} and Remark \ref{Rem-kpp},  yields 
 \begin{equation*}
\lim_{t\to +\infty} \sup_{ x\geq ct} v(t, x) =0,  \quad \forall c> c^*_v.
 \end{equation*}

\vspace{1em}
\noindent \textbf{Lower estimate: }
Let us  prove  the lower estimate of spreading speed of $v$.    From Assumption \ref{Ass-fg} and Lemma \ref{Lem-bound}, there exists a constant $L>0$ such that   
$$g(z,u,v)  \geq  g(z,1,0) - L(1-u) -Lv, \quad \forall z\in \R,  \;\; (u, v) \in [\beta, 1] \times [0, \overline{V}]. $$
 Lemma \ref{Lem-estimate} implies that  for each  $\alpha >0$ small  enough, there exist  some $M_{\alpha}>0$ and $T_{\alpha}>0$ such that $v$ satisfies 
\begin{equation*}
\partial_{t} v \geq \partial_{xx} v+ v \bigg( g(x-c_1t ,1,0) -L \alpha - LM_\alpha v -Lv\bigg),  \;\; \forall t\geq T_{\alpha},   x\in \R.
\end{equation*}
We denote  
$$r_{\alpha} (x- c_1 t):=  g(x-c_1 t, 1,0) - L \alpha   \;\; \text{ and }  \;\;    h_{\alpha}: = L M_{\alpha}   +L.$$
Let $w_\alpha (t,x)$  solve
\begin{equation} \label{subv}
\partial_{t} w_\alpha= \partial_{xx}w_\alpha+ w_\alpha \bigg( r_{\alpha} \big(x-c_1 (t+T_{\alpha}) \big) - h_\alpha w_\alpha \bigg),  \;\; \forall t>0,  x\in \R, 
\end{equation}
equipped with  initial data $w_\alpha (0, x)$, where   $w_\alpha (0, \cdot)$   is nonnegative, nontrivial,  continuous and compactly supported on  $\R$.   By Lemma \ref{Lem-bound} and the strong maximum principle for the $v$-component of the equation  \eqref{Pb},   we have $v(t,x)>0$ for all $t>0$  and $x\in \R$.   We can choose $w_\alpha (0, x)$ small enough such that $w_\alpha(0, \cdot) \leq v(T_\alpha, \cdot)$. The comparison principle implies that $w_{\alpha}(t,x)\leq v(t+T_{\alpha}, x)$ for all $t\geq 0$ and $x\in \R$.

Now we introduce an eigenvalue problem. Define 
\begin{equation*}
r^{\pm}_\alpha :=   r_{\alpha} (\pm \infty) = g(\pm \infty, 1,0) -L\alpha. 
\end{equation*}
Let $\Lambda_{1, \alpha} $ be the generalized principal eigenvalue of following eigenvalue problem:
\begin{equation*}
\phi_{yy}+  r_{\alpha} (y)  \phi = \Lambda \phi,  \quad \forall y\in \R.
\end{equation*}
Equivalently,
\begin{equation*}
\Lambda_{1, \alpha}  =  \inf\left\{ \Lambda \in \R: \;\;  \exists \psi \in C^2_{loc} (\R), \psi>0,  \psi_{yy} +  r_\alpha (y)  \psi \leq  \Lambda \psi \text{ in } \R  \right\}.
\end{equation*}
 Let us choose $\alpha>0$ small enough such that  $r^{\pm}_{\alpha}>0$ and $\inf_{y\in \R} r_ \alpha (y) >0$. 
 For each sufficiently small $\alpha >0$,     applying Theorem  \ref{Thm-KPP} and Remark \ref{Rem-kpp} to \eqref{subv} with \(g(\cdot)\) replaced by \(r_\alpha(\cdot)\), we obtain that 
\begin{equation*}
c^*_\alpha= \begin{cases}
2\sqrt{r_{\alpha}^+}, & \text{ if } c_1 \leq 2 \sqrt{r_\alpha^+}, \\
c_1,   & \text{ if } 2\sqrt{r_{\alpha}^+} \leq c_1 \leq  2 \sqrt{\Lambda_{1,\alpha} }, \\
\frac{c_1}{2} - \sqrt{\Lambda_{1,\alpha}- r_{\alpha}^{-} }  + \frac{r_{\alpha}^{-}}{\frac{c_1}{2} - \sqrt{\Lambda_{1,\alpha}- r_{\alpha}^{-} }  },   \quad  & \text{ if }   2 \sqrt{\Lambda_{1,\alpha} }  \leq  c_1 \leq 2  \left(  \sqrt{r_{\alpha}^-} + \sqrt{\Lambda_{1,\alpha}-r_{\alpha}^{-}} \right),\\
2\sqrt{r_{\alpha}^-},    & \text{ if } c_1  \geq   2 \left(  \sqrt{r_{\alpha}^-} + \sqrt{\Lambda_{1,\alpha}-r_{\alpha}^{-}} \right).
\end{cases}
\end{equation*}
By the  comparison principle for \eqref{subv},  we obtain that for each given $\alpha>0$, 
\begin{equation}\label{lim}
\liminf_{t\to +\infty} \inf_{ 0 \leq x\leq ct} v(t+T_\alpha, x) \geq  \liminf_{t\to +\infty} \inf_{ 0 \leq x\leq ct} w_\alpha  (t , x) >0,  \quad \forall c\in (0, c^*_\alpha).
\end{equation}
The time shift  $T_{\alpha}$   can be removed   since  the  arbitrariness of $c\in (0, c^*_{\alpha})$.
Thus, for all $\alpha>0$ sufficiently small,  we have 
\begin{equation}\label{lim2}
\liminf_{t\to +\infty} \inf_{ 0 \leq x\leq ct} v(t, x)  >0,  \quad \forall c\in (0, c^*_\alpha).
\end{equation}

\vspace{1em}
Finally, let $\alpha\to 0^+$. Note that  $r_\alpha^\pm$ is continuous with respect to $\alpha \geq 0$ and $r_\alpha^\pm \to r^\pm$ as $\alpha \to 0^+$.  Here $r^{\pm}$ is defined in \eqref{rrr}.    From the definition of generalized principal eigenvalue, we have 
 $$   |\Lambda_{1, \alpha} - \Lambda_{1}  |  \leq \| r_{\alpha}(\cdot) -r(\cdot) \|_{L^\infty(\R)} =L \alpha \to 0   \;\;   \text{ as } \alpha \to 0^+, $$
where $r(\cdot):= g(\cdot, 1,0)$ and $\Lambda_1$ is defined in \eqref{ev}.
Since \(c_\alpha^*\to c_v^*\) as \(\alpha\to0^+\), for every \(c'\in(0,c_v^*)\) there exists \(\alpha'>0\) sufficiently small such that \(c'<c_{\alpha'}^*\).
   Hence, it follows from \eqref{lim2}  that we have  
\begin{equation*}
\liminf_{t\to +\infty} \inf_{ 0 \leq x\leq c t} v(t, x) >0,  \quad \forall c\in (0, c^*_v).
\end{equation*}
This completes the proof.
\end{proof}

\section{Asymptotic behavior and application}
\subsection{Proof of Corollaries \ref{Cor1} and \ref{Cor2}}
As a consequence of Theorem \ref{Thm1}, we can first establish the large-time behavior of solutions at the leading edge.

\begin{proof}[Proof of Corollary \ref{Cor1}]
By Theorem \ref{Thm1}, it  suffices to prove that
\begin{equation*}
\lim_{t\to +\infty} \sup_{ x\geq (c_v^* + \eta) t}  \left|  u(t,x) -1 \right | =0,  \quad  \forall \eta >0.
\end{equation*}
By  contradiction,  for each  $\eta >0$,   we suppose  that there exist sequences $\{ t_n\}_{n\geq 1}$ and $\{ x_n\}_{n\geq 1}$ with $x_n \geq (c_v^* +\eta ) t_n$ and $t_n \to +\infty$ as $n \to+\infty$  such that 
\begin{equation}\label{uless1}
 \limsup_{n\to +\infty}  u(t_n,x_n) <1. 
\end{equation}
  Denote $(u_n, v_n) (t,x) := (u, v) (t+t_n, x+x_n)$. By standard parabolic estimates, 
  we have 
  \begin{equation*}
  \lim_{n\to +\infty} (u_n, v_n)(t,x) = (\tilde{u}, \tilde{v})(t,x) \quad \text{ locally uniformly in }  (t,x) \in \R^2, 
  \end{equation*}
  and 
\begin{equation*}
\begin{cases}
\partial_t \tilde{u}= d \partial_{xx} \tilde{u}  +  \tilde{u}  \tilde{f} \big( x-c_1t, \tilde{u}, \tilde{v} \big),\\
\partial_t \tilde{v}=  \partial_{xx} \tilde{v} + \tilde{v} \tilde{g} \big( x-c_1t, \tilde{u}, \tilde{v} \big).  \\
\end{cases}
\end{equation*}
From the definition of spreading speed,  we have $\tilde{v} (0, 0)=0$. By the strong maximum principle, we obtain $\tilde{v} \equiv 0$. Then $\tilde{ u}$ satisfies 
\begin{equation*}
\partial_t \tilde{u}= d \partial_{xx} \tilde{u} + \tilde{u} \tilde{f} (x-c_1t, \tilde{u}, 0), \quad \forall (t,x) \in \R^2.
\end{equation*}
Note that $\tilde{u} \geq \beta $ from Lemma \ref{Lem-bound}.
By the Liouville-type Lemma \ref{Lem-liou}, we obtain $\tilde u\equiv 1$, which contradicts   
  \eqref{uless1}. The proof is complete.

\end{proof}

Next we prove Corollary \ref{Cor2}. The proof is close to  \cite[Theorem 2.6]{ducrot2019spreading} and \cite[Theorem 4]{LamLee2024pp}. 

\begin{proof}[Proof of Corollary \ref{Cor2}]
Let us first prove \eqref{22}.   By changing the coordinate $x\mapsto -x$, from Theorem \ref{Thm1}, we know that there exists $c_l>0$ such that
\begin{equation*}
\liminf_{t\to +\infty} \inf_{ -(c_l-\eta) t   \leq   x \leq  0 }  v(t,x) >0,  \quad \forall \eta >0.
\end{equation*}
Choose $\delta\in (0, c_l)$. Then for all sufficiently  small  $\eta>0$, there exist   $T=T(\eta, \delta) >0$  and $\varepsilon_0>0$ such that 
\begin{equation*}
v(t,x) \geq \varepsilon_0     \quad \text{ for all } t\geq T   \text{ and }    - \delta t  \leq x\leq (c_v^* - \frac{\eta}{2}) t.
\end{equation*}
If  \eqref{22} fails,   then  there exist  $\varepsilon_1>0$  and  a sequence $\{( t_k, x_k)\}_{k \geq 1} $ with 
$0 \leq x_k   \leq ( c_v^*- \eta) t_k$ and $t_k \to +\infty$ as $k \to +\infty$ such that 
\begin{equation}\label{contr-uv}
 \left| u(t_k, x_k)-u^*_{-}  \right| + \left| v(t_k, x_k) - v^*_{-}\right|   \geq \varepsilon_1   \quad \forall k\geq 1.
\end{equation}
 Denote $(u_k, v_k) (t,x) := (u, v) (t+t_k, x+x_k)$. 
 Then, up to a subsequence,  we have 
 \begin{equation*}
 v_k(t, x)\geq \varepsilon_0,   \qquad      \forall (t,x )\in  \Omega_k,
 \end{equation*}
 where 
 \begin{equation*}
 \Omega_k: = \{  (t,x): t+t_k \geq T,  \;\; - \delta (t+ t_k)  \leq  x+x_k\leq (c_v^* - \eta/2) (t+t_k)  \}.
 \end{equation*}
 For every compact set \(K\subset\mathbb R^2\), we have \(K\subset\Omega_k\) for all sufficiently large \(k\).
From  Lemma \ref{Lem-bound}, we have 
\begin{equation}\label{uvbbb}
 0<\beta \leq  u_k(t,x) \leq 1   \text{ and }    \varepsilon_0  \leq   v_k(t,x) \leq \overline{V},    \quad \forall (t,x) \in \Omega_k,  
\end{equation}
where $\beta$ and $\varepsilon_0>0$ are independent of $k \geq 1$.
 By  parabolic estimates,   possibly along a subsequence,  we have 
  \begin{equation*}
  \lim_{k\to +\infty} (u_k, v_k)(t,x) = (\tilde{u}, \tilde{v})(t,x) \quad \text{ locally uniformly in }  (t,x) \in \R^2. 
  \end{equation*}
  Since $c_1\geq  2\sqrt{r^+}$,  we have $c_v^* \leq c_1$. Then we  obtain
    \[
    x_k-c_1t_k\le (c_v^*-\eta-c_1)t_k\le -\eta t_k\to-\infty  \quad \text{ as  } k \to \infty.
    \]
 Hence, 
  \begin{equation*}
  \begin{aligned}
   \lim_{k\to +\infty } f(x+x_k- c_1(t+t_k), u_k, v_k) = f( -\infty, \tilde{u}, \tilde{v})= f_{-}(\tilde{u}, \tilde{v}) ~~  \text{ locally uniformly in } (t,x) \in \R^2,\\
    \lim_{k\to +\infty } g(x+x_k- c_1(t+t_k), u_k, v_k) =g( -\infty, \tilde{u}, \tilde{v}) =g_{-}(\tilde{u}, \tilde{v})  ~~ \text{ locally uniformly in } (t,x) \in \R^2.\\
  \end{aligned}
  \end{equation*}
  Note that  $(\tilde{u}, \tilde{v})$ satisfies
\begin{equation*}
\begin{cases}
\partial_t \tilde{u}= d \partial_{xx} \tilde{u}  +  \tilde{u} f_{-} \left(   \tilde{u},  \tilde{v}\right),   & \forall (t,x) \in \R^2, \\
\partial_t \tilde{v}=  \partial_{xx} \tilde{v} + \tilde{v} g_{-} \left(   \tilde{u},  \tilde{v}\right),  & \forall (t,x) \in \R^2. \\
\end{cases}
\end{equation*}
The estimate  \eqref{uvbbb} yields 
\begin{equation*}
\beta \leq  \inf_{(t,x) \in \R^2} \tilde{u} \leq \sup_{(t,x) \in \R^2}  \tilde{u}(t,x)  \leq 1 \text{  and  }    \varepsilon_0 \leq  \inf_{(t,x) \in \R^2} \tilde{v} \leq \sup_{(t,x) \in \R^2}  \tilde{v}(t,x)  \leq \overline{V}.
\end{equation*}  
Moreover,  \eqref{contr-uv}  gives
\begin{equation}\label{contra22}
\left| \tilde{u}(0, 0)-u^*_{-}  \right| + \left| \tilde{v}(0, 0) - v^*_{-}\right|   >\varepsilon_1. 
\end{equation}
Since $\tilde{v} >0$, the strong maximum principle applied to  $\tilde{u}-$ equation implies $\tilde{u}<1$.  
Using  $d=1$, Assumption \ref{Ass-LF} and  Claim 7.1 in \cite{ducrot2019spreading}, we obtain
\begin{equation}\label{eq-lya}
 (\tilde{u}, \tilde{v}) (t,x) \equiv  (u^*_{-}, v^*_{-}).  
\end{equation}
This  contradicts  \eqref{contra22}   and   completes the proof of  \eqref{22}.

The proofs of \eqref{33} and \eqref{44} are very similar.  The only difference is the choice of the  sequences $\{(t_k, x_k)\}$ and hence the limiting system satisfied by $(\tilde{u}, \tilde{v})$.     
In the proof of  \eqref{33}, the sequence satisfies
\(x_k-c_1t_k\to-\infty\), and hence the limiting system is governed by \((f_-,g_-)\).
In the proof of \eqref{44}, the sequence satisfies
\(x_k-c_1t_k\to+\infty\), hence the limiting system is governed by  \((f_+,g_+)\).
The same Lyapunov argument then yields the desired contradiction.
The proof is complete.

\end{proof}

\subsection{The Lotka-Volterra prey-predator model}
We now illustrate the result with the following Lotka–Volterra prey–predator system:
\begin{equation}\label{Pb-ex}
\begin{cases}
\partial_t u=d\partial_{xx}u+au(1-u-pv), & t>0,\ x\in\mathbb R,\\
\partial_t v=\partial_{xx}v+v(b(x-c_1t)u-1-v), & t>0,\ x\in\mathbb R,\\
u(0,x)=u_0(x),\quad v(0,x)=v_0(x), & x\in\mathbb R,
\end{cases}
\end{equation}
where \(u\) and \(v\) denote the densities of the prey and the predator, respectively. The constant \(d>0\) is the diffusivity of the prey, \(a>0\) is the intrinsic growth rate of the prey, and \(p>0\) measures the predation rate. The function \(b(x-c_1t)\) describes the conversion rate of consumed prey into predator biomass. When \(b\) is monotone, the spreading speed of  \eqref{Pb-ex}  was studied in \cite{LamLee2024pp}.

\begin{example}
\rm 
Let \(d,a,p\) be positive constants and let \(b:\mathbb R\to (0,\infty)\) be bounded and uniformly continuous. Set
\[
b_{\sup}:=\sup_{z\in\mathbb R}b(z),\qquad r(z):=b(z)-1.
\]
Assume that the limits 
$$r^\pm:=\lim_{z\to\pm\infty}r(z)$$
exist and that $\inf_{z\in\mathbb R}r(z)>0$.
We require 
\[
\beta_0:=1-p(b_{\sup}-1)>0.
\]
Let
\[
\overline{V}:=b_{\sup}-1.
\]
Then
\[
\sup_{z\in\mathbb R}g(z,1,\overline{V})
=
\sup_{z\in\mathbb R}\{b(z)-1-\overline{V}\}\le 0.
\]
Moreover, choosing \(0<\delta<\beta_0\), we obtain
\[
\inf_{\substack{z\in\mathbb R\\0\le u\le \delta}}
f(z,u,\overline{V})
=
a\inf_{0\le u\le \delta}\{1-u-p\overline{V}\}>0.
\]
Hence Assumption 2.3 is satisfied.

Let the initial data \(u_0,v_0\in C_b(\mathbb R)\) satisfy
\[
0<\inf_{x\in\mathbb R}u_0(x)\le u_0(x)\le 1,\qquad
0\le v_0(x)\le \overline{V},\qquad v_0\not\equiv 0,
\]
and assume that \(v_0\) is compactly supported. Then Assumptions 2.1, 2.3 and 2.4 are satisfied. Therefore Theorem 2.7 applies to \eqref{Pb-ex}, and the spreading speed of the predator is given by the formula in Theorem 2.7 with
\[
g(z,1,0)=b(z)-1=r(z).
\]

In particular, if  $ r^+ < \Lambda_1,$
then the locking phenomenon occurs for
\[
2\sqrt{r^+}<c_1<2\sqrt{\Lambda_1}.
\]
If
\[
2\sqrt{\Lambda_1}
<
2(\sqrt{r^-}+\sqrt{\Lambda_1-r^-}\;),
\]
then the nonlocal pulling phenomenon occurs for
\[
2\sqrt{\Lambda_1}
< c_1 <  2( \sqrt{r^-} + \sqrt{\Lambda_1- r^-} \; ).
\]
Furthermore, let
\begin{equation*} 
b_\pm:=b(\pm\infty),\qquad
u_\pm^*:=\frac{p+1}{pb_\pm+1},\qquad
v_\pm^*:=\frac{b_\pm-1}{pb_\pm+1}.
\end{equation*}
Then \((u_\pm^*,v_\pm^*)\) are the unique positive equilibria of the limiting Lotka-Volterra systems
\[
\begin{cases}
u_t=au(1-u-pv),\\
v_t=v(b_\pm u-1-v).
\end{cases}
\]
For \(i\in\{+,-\}\), define
\[
\Phi_i(u,v)
:=
b_i\int_{u_i^*}^{u}\frac{\xi-u_i^*}{\xi}\,d\xi
+
ap\int_{v_i^*}^{v}\frac{\eta-v_i^*}{\eta}\,d\eta,
\qquad u>0,\ v>0,
\]
where \(b_+=b(+\infty)\) and \(b_-=b(-\infty)\). Then \(\Phi_i\) is strictly convex and attains its minimum at \((u_i^*,v_i^*)\). A direct computation shows that 
\[
\begin{aligned}
(uf_i(u,v),vg_i(u,v))\cdot\nabla\Phi_i(u,v)
&=
b_i(u-u_i^*)f_i(u,v)
+
ap(v-v_i^*)g_i(u,v)\\
&=
-ab_i(u-u_i^*)^2-ap(v-v_i^*)^2\\
&\le 0,
\end{aligned}
\]
where
\[
f_i(u,v)=a(1-u-pv),\qquad g_i(u,v)=b_i u-1-v.
\]
Thus the limiting ODE systems admit strict Lyapunov functions.

Consequently, the convergence conclusions in Corollary 2.13 apply to \eqref{Pb-ex}. Moreover, for this particular Lotka-Volterra model, the restriction \(d=1\) is not needed. Indeed, the Lyapunov function is separable:
\[
\Phi_i(u,v)=\Psi_{1,i}(u)+\Psi_{2,i}(v).
\]
The same argument used  in the proof of Lemma 4.1 in \cite{ducrot2021asymptotic}  can be applied to  prove that the limiting entire solutions are  equal to the coexistence states  for every $d>0$.
\end{example}

\section*{Acknowledgements}
The author thanks the anonymous reviewer for the helpful comments and suggestions.
This work  is partially supported by NSFC (No. 12401225, No. 12471150) and  the Strategic Priority Research Program of CAS (No. XDB0900100). 

\begin{small}
	\bibliography{jzc}
		\bibliographystyle{abbrv}
\end{small}



\end{document}